\documentclass[a4paper,oneside]{amsart}
\usepackage{amssymb,amsmath}
\usepackage{mathrsfs}

\usepackage{typearea}
\newcounter{lemmacounter}
\newcounter{thmcounter}
\newtheorem{lemma}[lemmacounter]{Lemma}
\newtheorem{corollary}[thmcounter]{Corollary}
\newtheorem{theorem}[thmcounter]{Theorem}

\def\ord{\mathop{\rm ord}\nolimits}

\newcommand{\places}[1]{M_{#1}}
\newcommand{\heightS}{h}
\newcommand{\height}[1]{\heightS(#1)}

\newcommand{\dmax}[2]{\max \{#1,#2\}}

\newcommand{\hproj}[1]{h_p{(#1)}}

\newcommand{\IR}{{\mathbb R}}

\newcommand{\IC}{{\mathbb C}}
\newcommand{\IZ}{{\mathbb Z}}

\newcommand{\IQ}{\mathbb{Q}}
\newcommand{\IQbar}{\overline{\IQ}}

\newcommand{\loccit}{{\it loc.cit.}}
\newcommand{\ssm}{\smallsetminus}

\renewcommand{\subset}{\subseteq}
\renewcommand{\supset}{\supseteq}

\begin{document}


\title{Quasi-equivalence of Heights and Runge's Theorem}
\author{P. Habegger}
\address{Department of Mathematics and Computer Science, University of Basel, Spiegelgasse 1, 4051 Basel,
Switzerland}
\email{philipp.habegger@unibas.ch}
\begin{abstract}
  Let $P$ be a polynomial that depends on
 two variables $X$ and $Y$ and has algebraic
  coefficients. If $x$ and $y$ are algebraic numbers with $P(x,y)=0$, 
then by work of N\'eron $h(x)/q$ is asymptotically equal to $h(y)/p$ where 
$p$ and $q$ are the partial degrees of $P$ in $X$ and $Y$,
respectively. In this paper we compute a completely explicit bound 
for $|h(x)/q-h(y)/p|$ in terms of $P$ which grows asymptotically as
$\max\{h(x),h(y)\}^{1/2}$.  We apply this bound to obtain a simple
version of Runge's Theorem on the integral solutions of certain polynomial
equations. 
\end{abstract}
\keywords{Heights, Absolute Siegel Lemma, Runge's Theorem}
\subjclass[2010]{Primary:  11G50,
secondary: 11D41, 11G30, 14H25, 14H50}
\maketitle


\section{Introduction}

Suppose $P$ is an irreducible polynomial in two variables $X$ and $Y$  and whose coefficients
are in $\IQbar$, an  algebraic closure of $\IQ$. If $x$ and $y$ are algebraic numbers with
$P(x,y)=0$ we investigate the relation between the absolute logarithmic Weil heights 
 $h(x)$ and $h(y)$; this height
is defined in Section \ref{sec:AuxFunction}.

Say  $p=\deg_X P \ge 1$ and $q=\deg_Y P \ge 1$. By
 work of N\'eron \cite{Neron} 
there exists a
constant $c(P)$ such that if $x$ and 
$y$ are algebraic numbers with $P(x,y)=0$,
then 
\begin{equation}
\label{SiegelIneq3}
\left|\frac{\height{x}}{q}-\frac{\height{y}}{p}\right| \le c(P)\max\left\{\frac{\height{x}}{q},\frac{\height{y}}{p}\right\}^{1/2}.
\end{equation}
See
Corollary 9.3.10 in Bombieri and Gubler's book \cite{BG} or
 Theorem B.5.9 in Hindry and Silverman's book \cite{DG2000} for a 
highbrow approach to this bound. 
One sometimes says that $h(x)/q$ and $h(y)/p$ are quasi-equivalent.

Our aim is to determine
an admissible constant $c(P)$ which is completely explicit in terms of $P$.
We will strive for a good dependency in the projective height
$\hproj{P}$
 of $P$, which we also define in Section \ref{sec:AuxFunction},
and  the partial degrees
 $p$ and $q$. 

\begin{theorem} 
\label{QuasiequivThm}
Let $P\in\IQbar[X,Y]$ be irreducible with $p=\deg_X P\ge 1$ and
  $q=\deg_Y P\ge 1$. 
If $P(x,y)=0$ with $x,y\in\IQbar$, then
(\ref{SiegelIneq3}) holds with  
\begin{equation*}
  c(P)  = 5\left(\log\left(2^{\min\{p,q\}}(p+1)(q+1)\right) +\hproj{P}\right)^{1/2}.
\end{equation*}
\end{theorem}

 Abouzaid \cite{Abouzaid} proved
a related height estimate.
In his bound, the dependence on the partial
degrees and the numerical constants are slightly worse.
Quasi-equivalence of heights also follows from 
Bartolome's Theorem 1.3 \cite{Bartolome:15}, but again  with larger numerical
constants
and worse dependency on $p$ and $q$. 

It is essential that $P$ is irreducible. For example both partial degrees
of $(X^2-Y)(X-Y^2)$  equal $3$. But (\ref{SiegelIneq3}) cannot
hold for this polynomial as $h(x^2)=2h(x)$ for all algebraic $x$. It is possible to formulate
a version of Theorem \ref{QuasiequivThm} when $K\subset\IQbar$ is a number field
and if  $P\in K[X,Y]$ is irreducible. In this case 
$P$ is up to a scalar 
factor the product of polynomials which are irreducible in
$\IQbar[X,Y]$ and conjugated over $K$. Thus all factors have equal partial
degrees. 

Sometimes it is useful to bound $\height{y}$ uniformly in terms of
$\height{x}$ if the height of $x$ is large. We do this in the
following corollary. 

\begin{corollary}
\label{QuasiequivCor} Let $P,p,$ and $q$ be as in Theorem \ref{QuasiequivThm}.
If $P(x,y)=0$ with $x,y\in\IQbar$ and 
\begin{equation}
\label{QuasiequivCorHypo}
\max\left\{\frac{h(x)}{q},\frac{h(y)}{p}\right\}
\ge  100  \left(\log\left(2^{\min\{p,q\}}(p+1)(q+1)\right)+\hproj{P}\right)
\end{equation}
then $\height{y}\le 2\frac pq \height{x}$.
\end{corollary}

The proof of Theorem \ref{QuasiequivThm} 
depends 
on the 
 theory of  functions fields
and on the Absolute Siegel Lemma
by Zhang \cite{ZhangArVar}.
 Roy and Thunder's \cite{RTAbsSiegel,RTAbsSiegelErrat} absolute Siegel Lemma
would also suffice but would lead to different numerical constants. 
Using 
 Bombieri and Vaaler's 
classical version of Siegel's Lemma instead would come at the cost of introducing
a dependency in $c(P)$ on a number  field containing the coefficients
 of $P$.

We give a short sketch of the proof of Theorem \ref{QuasiequivThm}.
Let $m$ and $n$ be large integers such that $n/m$ is approximately 
equal to $p/q$. In Section \ref{sec:AuxFunction} 
we use the Absolute Siegel Lemma to construct
polynomials $A$ and $B$ in $X$ and $Y$ with algebraic coefficients of
bounded height, not both zero, such that $P$ divides $AY^m - B$ as a
polynomial. By choosing the parameters appropriately, we can arrange
that $P\nmid A$. 
Say $P(x,y)=0$, then $A(x,y) y^m=B(x,y)$. If we assume for the moment
$A(x,y)\not=0$, then we may bound the height of $y$ in terms of the height of
$x$ by using the product formula.
In Section \ref{sec:zerobounds} 
we show that a suitable vanishing order
cannot be too large.
We then apply an   appropriate differential operator and replace $A,B$ by new 
polynomials $A',B'$ with controlled projective height and degree
such that $A'(x,y)y^m=B'(x,y)$ and $A'(x,y)\not=0$.
Thus again we get a  bound for $h(y)$ in terms of $h(x)$. 
By swapping $x$ and $y$ we get an estimate in the other direction and
 this completes the
proof if $(x,y)$ is not singular point on the vanishing locus of $P$.
Singular points can be handled directly. 
A novel aspect of our approach is that it does
 not depend on
 Eisenstein's Theorem 
which bounds
the coefficients of a power series of an algebraic function.
An explicit version of 
Eisenstein's Theorem was used in  the work of Abouzaid and Bartolome.

 Runge \cite{Runge} proved that $P(x,y)=0$ 
admits only finitely many solutions $(x,y)\in \IZ^2$
if $P\in\IQ[X,Y]$ is irreducible 
with $\deg_X P = \deg_Y P = \deg P$ and if the homogeneous part of $P$ 
of maximal degree is not a rational multiple of 
 the power of an irreducible 
polynomial in $\IQ[X,Y]$.
Runge's method is effective and  explicit upper bounds for
$\max\{|x|,|y|\}$ were obtained for example by Hilliker-Straus  
\cite{HillikerStraus}
and Walsh \cite{Walsh,WalshCorr}. They  rely
on  Eisenstein's Theorem.

We
will prove a simple and explicit version of Runge's Theorem using  
 Theorem \ref{QuasiequivThm}.

\begin{theorem} 
\label{RungeThm}
Let $P\in\IZ[X,Y]$ be irreducible in $\IQbar[X,Y]$ and
assume $d=\deg_X P = \deg_Y P = \deg P$. Furthermore, assume that the
homogeneous part of $P$ of degree $d$ is not a rational
multiple of the power of an
irreducible polynomial in $\IQ[X,Y]$. If $P(x,y)=0$ with $x,y\in\IZ$, then
\begin{equation*}
\log\max\{1,|x|,|y|\} \le 115 d^4 (\log(2d)+\hproj{P}). 
\end{equation*}
\end{theorem}

 Walsh's \cite{Walsh,WalshCorr} result holds for a larger
class of polynomials but our dependency on the degree is $d^4$ instead
of his $d^6$.

It would be interesting to see if a 
 more sophisticated version of Runge's Theorem, such as
 Bombieri's on page 304 \cite{BombieriWeilDecomp},
can be proved using our Theorem \ref{QuasiequivThm}.

A variation of this paper appeared in the appendix of the
author's 2007 Ph.D. thesis. 
He thanks his supervisor David Masser for support throughout those years.
He is also grateful to Umberto Zannier for pointing out Zhang's
version
of the Absolute Siegel Lemma and to the referee for helpful comments
and corrections.

\section{Construction using the Absolute Siegel Lemma}\label{sec:AuxFunction}

We begin by setting up notation, our reference for heights
is Chapter 1.5 in Bombieri and Gubler's book \cite{BG}. 

Let $K$ be a number 
field. A place $v$ of $K$ is an absolute value $|\cdot|_v:K\rightarrow
[0,\infty)$ such that either
  $|x|_v=\dmax{x}{-x}$ for all $x\in\IQ$ 
or $|\cdot|_v$ coincides with the $p$-adic absolute value on $\IQ$ for
a prime $p$ and $|p|_v=1/p$. 
In the former case we call $v$ infinite and in the latter we call it finite.
If $v$ is infinite, then $|x|_v = |\sigma(x)|$ 
for  a ring homomorphism $\sigma:K\rightarrow\IC$
that is uniquely determined up-to complex conjugation. We set $d_v=1$
if 
$\sigma(K)\subset\IR$ and $d_v=2$ else wise. 
A finite place $v$ is induced by a maximal ideal in the ring of
integers of $K$. 
We set  $d_v$ to be  the product of the ramification index and the
residue degree attached to this prime ideal. 
We let $\places{K}$ denote the set of all
places of $K$. 

The absolute logarithmic Weil height of $x\in K$ is
\begin{equation}
\label{def:weilheight}
  h(x) = \frac{1}{[K:\IQ]}\sum_{v\in\places{K}} d_v \log \max\{1,|x|_v\}.
\end{equation}
It is well-defined and attains the same value at $x$ when evaluated
using any number
field $F\supset K$. 

For $a\in K^N$ we define
$|a|_v$ to be the maximum of the absolute values of the coordinates of
 $a$ with respect to $v$. 
If $P$ is a polynomial in any number of variables with coefficients in
$K$, then $|P|_v$ denotes
the maximum of the absolute values of the coefficients of
 $P$ with respect to $v$. 
The projective
height of $a\not=0$ is 
\begin{equation}
\label{def:projheight}
\hproj{a} = \frac{1}{[K:\IQ]} \sum_{v\in\places{K}} d_v \log |a|_v.
\end{equation}
By the product formula, cf. Proposition 1.4.4 \cite{BG}, $\hproj{a}$ is invariant under replacing $a$ by a non-zero scalar
multiple of itself, so $\hproj{a}\ge 0$. 
If $P$ is a non-zero polynomial with algebraic coefficients we 
set $\hproj{P}$ to be the projective height of the vector whose
coordinates are the non-zero coefficients of $P$. 
If $Q$ is a further polynomial with algebraic coefficients it will
 be useful to set $\hproj{P,Q} = \hproj{P+TQ}$ 
where $T$ is an unknown that does not appear in $P$ or $Q$. 

For any place $v$ of $K$ and integer $n\ge 1$ it is convenient to define
$\delta_v(n) = \max\{1,|n|_v\}$.

We start by proving a lemma concerning simple properties of
places and heights. It corresponds to Lemma  A.1 in the author's thesis, whose part (i) is
incorrect. 

\begin{lemma}
\label{AuxLemma}
Let $K$ be a field and $A,B\in K[X,Y]$.
\begin{enumerate}
\item[(i)] If $K$ is a number field and if $v\in\places{K}$, then
\begin{alignat}1
\nonumber
|A+B|_v &\le \delta_v(2)\max\{|A|_v,|B|_v \}, \\
\nonumber
|AB|_v &\le \delta_v((\min\{\deg_X A,\deg_X B\}+1)(\min\{\deg_Y A,\deg_Y B\}+1))|A|_v |B|_v.
\end{alignat}
\item[(ii)] If $K=\IQbar$ and $x,y\in\IQbar$ with $A(x,y)\not=0$, then
\begin{alignat}1
\nonumber
\height{&B(x,y)/A(x,y)} \le \hproj{A,B} + \max\{\deg_X A,\deg_X B\}\height{x}
\\ \nonumber
&+ \max\{\deg_Y A,\deg_Y B\}\height{y} \\ \nonumber
&+ \log\max\{(\deg_X A+1)(\deg_Y A+1),
(\deg_X B+1)(\deg_Y B+1)\}
\end{alignat}
\item[(iii)] Say $K=\IQbar$ and $x,y\in\IQbar$ with $A(x,y)=0$. 
If $A$ is not divisible in $\IQbar[X,Y]$ by any $X-\alpha$ with
 $\alpha\in\IQbar$, then
\begin{equation}
\label{AuxLemmaPart3}
\height{y} \le \hproj{A}+ (\deg_X A) \height{x} + \log((\deg_X A+1)\deg_Y A).
\end{equation}
\end{enumerate}
\end{lemma}
\begin{proof}The first inequality in (i) follows  from the triangle
  inequality. To prove the second inequality we write $A=\sum_{i,j}
  a_{ij}X^iY^j$ and $B=\sum_{i,j}b_{ij}X^iY^j$ . Then
  $AB=\sum_{i,j}c_{ij}X^iY^j$ with 
\begin{equation}
\nonumber
c_{ij} = \sum_{\substack{i'+i''=i \\ j'+j''=j}} a_{i'j'}b_{i''j''}.
\end{equation}
The sum above involves at most 
$(1+\min\{\deg_X A,\deg_X B\})(1+\min\{\deg_Y A,\deg_Y B\})$ non-zero terms. 
Hence the desired
inequality follows from the triangle inequality.

Now we prove part (ii). The product formula
 implies
\begin{equation}
\label{BAbound}
\height{B(x,y)/A(x,y)}
= \frac{1}{[F:\IQ]} \sum_{v\in \places{F}} d_v \log \max\{|A(x,y)|_v,|B(x,y)|_v\}
\end{equation}
where $F$ is a number field containing $x$, $y$ and the coefficients of $A$
and $B$.
Note that the polynomial $A$ involves at most $(\deg_X A+1)(\deg_Y A+1)$
non-zero coefficients, hence the triangle inequality gives
\begin{alignat}1
\nonumber
|A&(x,y)|_v \le \\ \nonumber &\delta_v((\deg_X A+1)(\deg_Y A+1))|A|_v 
\max\{1,|x|_v\}^{\deg_X A}\max\{1,|y|_v\}^{\deg_Y A}
\end{alignat}
for each $v\in \places{F}$.
Of course a similar inequality holds for $|B(x,y)|_v$. These inequalities
inserted into (\ref{BAbound}) conclude this part of the lemma.

Our proof for Part (iii) follows the lines of Proposition 5 \cite{BMSprindzhuk}.
Say $A=a_q Y^q+\cdots +a_0$ with $a_i\in\IQbar[X]$ and $a_q\not=0$,
so $q=\deg_Y A$. 
By hypothesis
there exists a maximal $q'\ge 1$ such that $a_{q'}(x)\not=0$.
Let $F$ be a number field that
contains $x,y$, and the coefficients of $A$. 
If $v\in\places{F}$,  then 
\begin{equation}
\nonumber
|a_{q'}(x) y^{q'}|_v \le \delta_v(q') \max_{0\le k \le q'-1} \{|a_k(x)|_v\} 
\max\{1,|y|_v\}^{q'-1}
\end{equation}
and so
\begin{equation}
\nonumber
\max\{1,|y|_v\} \le \delta_v(q') \max_{0\le k\le q'}
\{|a_k(x)|_v / |a_{q'}(x)|_v \}.
\end{equation}
We use the last inequality, $q'\le q$, the product formula, and (\ref{def:weilheight}) to deduce
\begin{equation}
\label{auxlemmahybound}
\height{y} \le \log q + \frac{1}{[F:\IQ]}
\sum_{v\in\places{F}} d_v \log \max_{0\le k\le q}\{|a_k(x)|_v\}.
\end{equation}
The triangle inequality implies 
\begin{equation}
\nonumber
|a_k(x)|_v\le \delta_v(\deg_X A+1)\max\{1,|x|_v\}^{\deg_X A}|A|_v
\end{equation}
 for any $v\in\places{F}$.
 We apply this inequality to (\ref{auxlemmahybound}) to complete
the proof.
\end{proof}

We will sometimes apply property (iii) of the previous lemma to a non-zero
$A\in\IQbar[Y]$ and $x=0$. Inequality (\ref{AuxLemmaPart3}) then
reduces to $\height{y}\le \hproj{A}+\log (\deg_Y A)$.

We  introduce a basic notion of a sparsity. If
$\mathcal{A}=(a_{ij})$ is an $M\times N$
matrix, then we set 
\begin{equation}
\nonumber
S(\mathcal{A}) = \max_{1\le i\le M}\#\{j : a_{ij}\not = 0\}.
\end{equation}
If $\mathcal{A}$ has  coefficients in a number field $K$ and is
non-zero we define $\hproj{\mathcal{A}}$ to be the projective height of
$\mathcal{A}$ taken as an element of $K^{MN}\ssm\{0\}$,
as in (\ref{def:projheight}).

Next we adapt
Zhang's  Absolute Siegel Lemma, as it was used
by David and Philippon \cite{DPTores},
 to our notation. See also Bartolome's Section 2.4 \cite{Bartolome:15}. 

\begin{lemma}
\label{SiegelLemma}
Let $\mathcal{A} \in \text{Mat}_{MN}(\IQbar)$ have rank
$M<N$. Then there
exists $v\in\IQbar^{N}\ssm\{0\}$ such that $\mathcal{A}v=0$ and
\begin{equation}
\nonumber
\height{v} \le  
\frac{M}{N-M}\left(\frac 12 \log S(\mathcal{A}) + \hproj{\mathcal{A}}\right)
+\frac{\log(N-M)}{2}.
\end{equation}
\end{lemma}
\begin{proof}
Let $\epsilon > 0$. 
We apply David and Philippon's Lemme 4.7 \cite{DPTores}
by which there exists 
$v\in\IQbar^N\ssm\{0\}$ with $\mathcal{A}v=0$ and
\begin{equation*}
\height{v} \le \frac{1}{N-M} h(V) + \frac{1}{N-M} 
\sum_{i=1}^{N-M-1} \sum_{j=1}^i \frac{1}{2j}+\epsilon
\end{equation*}
where $V\subset\IQbar^N$ is the kernel of 
$\mathcal{A}$ and $h(V)$ is the logarithmic height of the
vector space $V$ as defined just before Lemme 4.7  \cite{DPTores}. 
We have
$(N-M)^{-1}\sum_{i=1}^{N-M-1}\sum_{j=1}^i (2j)^{-1} < 2^{-1}\log(N-M)$, 
so 
\begin{equation}
  \label{SiegelLemmaIneq2}
h(v) \le \frac{1}{N-M}h(V)+\frac{\log{(N-M)}}{2}. 
\end{equation}
if $\epsilon>0$ is small enough. 
We note that the  height used by David and Philippon
 uses the Euclidean norm at the infinite places. It is at least as large
 as $h(v)$
 which uses the supremum norm at all places.
By Corollary 2.8.12  \cite{BG}
the height $h(V)$ is equals $h_{\mathrm{Ar}}(\mathcal{A}^t)$ as in
Remark 2.8.7 \loccit{}
where $\mathcal{A}^t$ is the transpose of $\mathcal{A}$.
In other words, $h(V)$ 
is  the height of the  vector in $\IQbar^{{N \choose M}}$ 
whose entries are
the determinants of all $M\times M$ minors of $\mathcal{A}$ 
with the Euclidean norm taken at
the infinite places and maximum norm at the finite places. 

By Fischer's Inequality, cf. Remarks and 2.8.9 and 2.9.8 \cite{BG}, 
we find $h(V) \le (M/2) \log S(\mathcal{A}) + M \hproj{\mathcal{A}}$;
here we used that each row of $\mathcal{A}$ contains
at most $S(\mathcal{A})$ non-zero entries.
\end{proof}

\begin{lemma}
\label{ConstructLemma1}
Let $P\in\IQbar[X,Y]$ with $p=\deg_X P \ge 1$ and $q=\deg_Y P \ge 1$. Furthermore,
let $m$ and $n$ be integers with $m\ge 2q+1$ and $n\ge p$. If
$t=q(n+1)-mp\ge 1$,  there exist $A,B\in\IQbar[X,Y]$ with 
$P\nmid A$, 
\begin{alignat}1
\label{ConstructionAssertion0}
AY^m-B \in P\cdot\IQbar[X,Y] \ssm\{0\}, \quad
\deg_X A,\deg_X B \le n,\quad \deg_Y A,\deg_Y B \le q-1,
\end{alignat}
and
\begin{alignat}1
\label{ConstructionAssertion1}
\hproj{A,B} \le \frac{m(n-p+1)}{t} \left(\log\left((p+1)(q+1)\right) + \hproj{P}\right) + \frac{\log(2nq)}{2}.
\end{alignat}
\end{lemma}
\begin{proof}
Let $\mathfrak{Q}=\sum_{k,l} q_{kl} X^k Y^l\in\IZ[X,Y,q_{kl}]$ 
with $\deg_X\mathfrak{Q} = n-p$,
 $\deg_Y\mathfrak{Q}=m-1$, and where the $q_{kl}$ are treated as
unknowns. We define  linear forms $f_{ij}\in
\IQbar[q_{kj}]$ for $0\le i \le n$, $0\le j\le m+q-1$ by
\begin{equation}
\nonumber
P \mathfrak{Q} = \sum_{i,j} f_{ij} X^i Y^j.
\end{equation}
Each  non-zero coefficient of $f_{ij}$ is a
 coefficient of $P$. So
\begin{equation}
\label{ConstructionEq}
f_{ij} = 0\quad (0\le i \le n, \quad q\le j \le m-1)
\end{equation}
is a system of linear equations of a certain rank $M$ in the
$N = (n-p+1)m$
unknowns $q_{kl}$. We have
\begin{equation}
\label{BoundNMBelow}
M \le (n+1)(m-q) = N - t.
\end{equation}
Because $N-M \ge t \ge 1$ there is a non-trivial solution. 
Any such solution gives rise to a non-zero polynomial $Q\in \IQbar[X,Y]$ 
such that the coefficients of $PQ$ satisfy (\ref{ConstructionEq}) 
and hence $PQ = AY^m-B$ for unique polynomials $A,B\in \IQbar[X,Y]$ with $\deg_X A,\deg_X B\le n$ and 
$\deg_Y A,\deg_Y B\le q-1$. The terms in $AY^m$ and $B$ do not overlap,
  hence $\hproj{A,B}= \hproj{PQ}$. 

The final
term in the upper bound (\ref{SiegelLemmaIneq2}) 
works against us if $N-M$ is large. We
now work out a lower bound for $M$.
A non-trivial $\IQbar$-linear combination of $X^k Y^lP$ where $0\le k\le n-p$ and
$q\le l\le m-q-1$ is not of the form $AY^m-B$ with $A$ and $B$ 
satisfying the degree bounds in (\ref{ConstructionAssertion0}). 
Recall that $m\ge 2q+1$, so 
$M \ge (n-p+1)(m-2q) \ge n-p+1 \ge 1$ and hence
\begin{equation}
\label{BoundNMAbove}
N-M \le (n-p+1)(m-(m-2q))\le 2nq.
\end{equation}

We will apply Siegel's Lemma to find a solution $Q$ with small
projective
height. 
We choose a subset of the linear forms $f_{ij}$ ($0\le i\le n$, $q\le
j\le m-1$) with rank
$M$ and use the coefficients of each such linear form to define a row in an 
$M\times N$ matrix $\mathcal{A}$.
The non-zero entries of $\mathcal{A}$ are coefficients of $P$,  
hence $\hproj{\mathcal{A}} \le \hproj{P}$. Furthermore, by definition each
$f_{ij}$ involves at most $(p+1)(q+1)$ non-zero coefficients and hence
$S(\mathcal{A}) \le (p+1)(q+1)$. 
By Lemma \ref{SiegelLemma} and our discussion above 
there exists a non-zero solution $Q\in \IQbar[X,Y]$ of 
(\ref{ConstructionEq}) that satisfies
\begin{equation}
\label{hQBound0}
\hproj{Q} \le  
\frac{M}{N-M}\left(\frac 12\log((p+1)(q+1))+ \hproj{P}\right) + \frac{\log(N-M)}{2}. 
\end{equation}

Lemma \ref{AuxLemma}(i)
implies $\hproj{PQ}\le   \log((p+1)(q+1))+\hproj{P}+\hproj{Q}$.
Furthermore, we use the inequalities (\ref{hQBound0}), (\ref{BoundNMBelow}), and
(\ref{BoundNMAbove}) to conclude that $\hproj{PQ}$ is at most
\begin{alignat}1
\nonumber
& \log((p+1)(q+1)) + \hproj{P}  + 
\frac{M}{t}\left(\log((p+1)(q+1)) + \hproj{P}\right) 
+ \frac{\log(2nq)}{2}
\\ \nonumber
&\quad = \frac{M+t}{t} \left(\log((p+1)(q+1)) + 
\hproj{P}\right)  +  \frac{\log(2nq)}{2}
\\ 
\nonumber &\quad\le\frac Nt \left(\log\left((p+1)(q+1)\right) +  \hproj{P}\right) + \frac {\log(2nq)}{2}.
\end{alignat}
This inequality completes the proof of (\ref{ConstructionAssertion1})
because $N= m(n-p+1)$.

Finally, we must verify $P\nmid A$.
Indeed assuming the contrary, then $P$ also divides $B$. Because $\deg_Y
A,\deg_Y B \le q-1$ we have $A=B=0$, a contradiction to
$AY^m-B\not=0$. 
\end{proof}

\section{Multiplicity Estimates}
\label{sec:zerobounds}

We need some  facts about function fields which we recall here for
the reader's convenience. We refer to
 Chevalley's book  \cite{IntroAlgFunctions} for proofs.

For a field $F$ we write $F^\times = F \ssm\{0\}$. 
Suppose $F$ contains an algebraically closed subfield $L$ and
 that
there exists an element $t\in F$ that is transcendental over $L$ 
such that $F$ is a finite field extension of $L(t)$. Then $F$ is a function field
over $L$. We define $\places{F}$ to be the set of the
maximal ideals of all the proper valuation rings of $F$ containing
$L$.
This set is the function field analogue of $\places{K}$
for a number field $K$.
Observe that its elements, the places of $F$, have degree $1$ since $L$ is
algebraically closed. 
We will identify an element of $\places{F}$ with the valuation function it
 induces.
Hence an element of $\places{F}$ is a
surjective map $v:F\rightarrow \IZ\cup\{\infty\}$ such that for all
$a,b\in F$ we have 
 $v(ab)=v(a)+v(b)$  and
 $v(a+b)\ge \min \{v(a),v(b)\}$, $v(a)=\infty$ if and
only if $a=0$, and
$v(a)=0$ if $a\in L^\times$; 
we use the convention $\infty + x = x+\infty = \infty$
and $\min \{\infty,x\}=\min\{x,\infty\}=x$ for all $x\in\IZ\cup\{\infty\}$.

If $a\in F^\times$, then $v(a)=0$ for all but finitely many $v\in \places{F}$ and 
\begin{equation}
\nonumber
\sum_{v\in \places{F}} v(a) = 0.
\end{equation}
Furthermore, if $a\in F\ssm L$ then 
\begin{equation}
\nonumber
\sum_{v\in \places{F}} \max\{0,v(a)\} = [F:L(a)].
\end{equation}

Suppose $P\in\IQbar[X,Y]$ is irreducible
 and let 
$F$  denote the  field of fractions of the domain
 $\IQbar[X,Y]/(P)$.
Then $F$ is a function field over $L=\IQbar$ .
By abuse of notation we shall consider
polynomials in $\IQbar[X,Y]$ as elements of $F$ via the quotient map. Note that
any polynomial in $\IQbar[X,Y]$ that is not divisible by $P$ maps to $F^\times$.

Let $\pi=(x,y)\in\IQbar^2$ with $P(\pi)=0$
such that $\frac{\partial P}{\partial X}, \frac{\partial P}{\partial Y}$
do not both
vanish at $\pi$, then we call $\pi$ a \emph{regular zero} of $P$.
Let us assume for the moment that
$\frac{\partial P}{\partial Y}(\pi) \not =0$,
then there exists a unique $v_\pi \in \places{F}$  with $v_{\pi}(X-x)=1$ 
and $v_{\pi}(Y-y)\ge 1$. Moreover, there exists
$E$ in $\IQbar[[T]]$, the ring of formal power series with
coefficients in $\IQbar$, 
such that $E(0)=0$ and $P(x+T,y+E)=0$. 
For any $A\in\IQbar[X,Y]$ not divisible by $P$ we have 
\begin{equation}
\nonumber
\ord A(x+T,y+E) = v_\pi (A)
\end{equation}
where $\ord$ is the standard valuation on $\IQbar[[T]]$. 
Therefore $v_\pi(A)\ge 1$ if and only if $A(x,y)=0$.
If
$\frac{\partial P}{\partial X}(\pi)\not=0$ then these properties
 hold with the roles of $X$ and $Y$ reversed. 

Let $A\in\IQbar[X,Y]$, we define
\begin{equation}
\nonumber
D(A) = \frac{\partial P}{\partial Y}\frac{\partial A}{\partial X}
-\frac{\partial P}{\partial X}\frac{\partial A}{\partial Y}
\in\IQbar[X,Y].
\end{equation}
We also set $D^0(A)=A$ and inductively $D^s(A) = D(D^{s-1}(A))$ for
all positive integers $s$. A formal verification yields
$D(P)=0$ and $D(AB) = D(A)B+A D(B)$ for all  $B\in \IQbar[X,Y]$. Thus we have Leibniz's
rule
\begin{equation}
\label{eq:leibniz}
D^s(AB) = \sum_{k=0}^s 
{s \choose k} D^k(A) D^{s-k}(B)
\quad\text{and}\quad
D^s(PA) = P D^s(A)\text{ if } s\ge 0. 
\end{equation}

\begin{lemma}
\label{DsBoundLemma}
Let $K$ be a number field and $P\in K[X,Y]$ with $p=\deg_X P \ge 1$ and
 $q=\deg_Y P \ge 1$.
Furthermore, assume
 $A\in K[X,Y]$ 
such that $\deg_X A \le n$, $\deg_Y A \le q-1$. 
Then for any non-negative $s\in\IZ$ we have
\begin{alignat}1
\label{QuasiEquivDerivBound}
\deg_X D^s(A) \le n+(p-1)s\quad\text{and}\quad \deg_Y D^s(A) \le (q-1)(s+1).
\end{alignat}
Moreover, if $r=\max\{p,q\}$ and $v\in\places{K}$ then
\begin{alignat}1
\label{DerivBound2}
\quad |D^s (A)|_v \le \delta_v(2(p+1)(q+1)r(n+rs))^s |P|_v^s |A|_v.
\end{alignat}
\end{lemma}
\begin{proof}
We note $\deg_X D(A)\le \deg_X(A) + p-1$ and so the first inequality in
(\ref{QuasiEquivDerivBound}) follows by induction on $s$. 
The second inequality is proved
similarly. 

We now show (\ref{DerivBound2}) by induction on $s$. 
The case $s=0$ being
trivial we may assume $s\ge 1$ and also $D^s(A)\not=0$. 
 For brevity set $|\cdot|=|\cdot|_v$.
We apply Lemma \ref{AuxLemma}(i) to deduce
\begin{equation}
\nonumber
|D^s(A)|\le \delta_v(2(p+1)(q+1))
\max\left\{\left|\frac{\partial P}{\partial Y}\right| 
\left|\frac{\partial D^{s-1}(A)}{\partial X}\right|,
\left|\frac{\partial P}{\partial X}\right| 
\left|\frac{\partial D^{s-1}(A)}{\partial Y}\right|\right\}.
\end{equation}
 By bounding the partial derivatives of the 
polynomials in the usual manner we get 
\begin{equation}
\nonumber
|D^s(A)|\le 
\delta_v(2(p+1)(q+1)r \max\{\deg_X D^{s-1}(A),\deg_Y D^{s-1}(A)\}) 
|P| |D^{s-1}(A)|.
\end{equation}
The inequalities in (\ref{QuasiEquivDerivBound}) imply
\begin{equation}
\nonumber
|D^s(A)| \le \delta_v(2(p+1)(q+1)r\max\{n+(p-1)(s-1),(q-1)s\}) |P| |D^{s-1}(A)|.
\end{equation}
The expressions inside the  maximum are bounded from above by $n+rs$. Applying the
induction hypothesis completes the proof.
\end{proof}

\begin{lemma}
\label{HopitalLemma}
Suppose $P\in\IQbar[X,Y]$ is irreducible,
let  $\pi\in\IQbar^2$ be a regular zero of $P$ and let
 $v=v_\pi\in \places{F}$ be the
valuation described above.
If $A\in\IQbar[X,Y]$ is not divisible by $P$ and
$A(\pi)=0$, then $D(A)$ is not divisible by $P$ and
\begin{equation}
\nonumber
v (D(A)) = v (A) - 1.
\end{equation}
\end{lemma}
\begin{proof}
We assume $\frac{\partial P}{\partial Y}(\pi)\not = 0$, the case
$\frac{\partial P}{\partial X}(\pi)\not = 0$ is similar. 
Say $\pi=(x,y)$. 
There exists $E\in T\IQbar[[T]]$ such that $P(x+T,y+E)=0$ and 
$v(A) = \ord A(x+T,y+E)\ge 1$. By the chain rule we have
\begin{equation*}
0 = \frac{d}{dT} P(x+T,y+E) = \frac{\partial P}{\partial X}(x+T,y+E) +
 \frac{dE}{dT}\frac{\partial P}{\partial Y}(x+T,y+E).
\end{equation*}
We use this  and the definition of $D$  to obtain
\begin{alignat*}1
&\ord D(A)(x+T,y+E) \\
&= \ord \left(\left( \frac{\partial P}{\partial
  Y}\frac{\partial A}{\partial X} - \frac{\partial P}{\partial
  X}\frac{\partial A}{\partial Y}\right)(x+T,y+E)\right) 
\\ \nonumber & = \ord \frac{\partial P}{\partial Y}(x+T,y+E) 
+ \ord \left( \frac{\partial A}{\partial X}(x+T,y+E) +
\frac{dE}{dT}\frac{\partial A}{\partial Y}(x+T,y+E)\right).
\end{alignat*}
By our assumption we have
$\ord \frac{\partial P}{\partial Y}(x+T,y+E) = 0$
which we insert into the equality above and 
use the chain rule again as well as $A(x,y)=0$ to get
\begin{equation}
\nonumber
\ord D(A)(x+T,y+E)  = \ord \frac{d}{dT} A(x+T,y+E)
= \ord A(x+T,y+E) - 1.
\end{equation}
Hence $v(D(A)) = v(A)-1$. In particular, $P$ does not divide $D(A)$. 
\end{proof}

We now prove a multiplicity estimate which will be useful later on.

\begin{lemma}
\label{ZeroBoundLemma}
Let $A,B,P,m,p,q,$ and $t$ be as in Lemma \ref{ConstructLemma1} with
$t\ge 1$. 
Furthermore, assume $P$ is 
irreducible and $\deg P=p+q$. If $\pi\in\IQbar^2$ is
a regular zero of $P$,  there exists an integer $s$ with 
$0 \le s \le t + pq - p -q$
such that $D^s (A)(\pi)\not =0$ 
and $D^k(A)(\pi)=0$ for all $0\le k < s$.
\end{lemma}
\begin{proof}
Let $F$ be as above Lemma \ref{DsBoundLemma}. For brevity set $v=v_\pi$. 
Clearly $X,Y\in F\ssm \IQbar$ since $p$ and $q$ are both positive; also 
$v(X),v(Y)\ge 0$. 
Furthermore, $A\not=0$ in $F$ by Lemma
\ref{ConstructLemma1}. 

We first claim
that for any $v'\in \places{F}$ at least one of the two $v'(X),v'(Y)$ is
non-negative. Indeed we argue by contradiction so 
let us assume $v'(X)<0$ and $v'(Y) <0$. Then for any integers $i,j$ with $0\le
i \le p$, $0\le j\le q$ and $i+j<p+q$ we have
\begin{equation}
\label{ZBLIneq1}
iv'(X)+jv'(Y) > pv'(X)+qv'(Y).
\end{equation}
Now by hypothesis $P=\alpha X^p Y^q + \tilde P$ with $\alpha\not=0$ and $\deg
\tilde P < p+q$. We apply the ultrametric inequality and (\ref{ZBLIneq1}) to get
\begin{equation}
\nonumber
pv'(X)+qv'(Y) \ge 
\min_{\substack{0\le i\le p,0\le j \le q \\ i+j<p+q}} 
\{ iv'(X)+jv'(Y)\}
> p v'(X) + q v'(Y),
\end{equation}
a contradiction. 

Now assume $v'\in \places{F}$ such that $v'(Y) < 0$. Then $v'(X) \ge 0$ by the
discussion above and 
\begin{equation}
\label{ZBLEq2}
v'(A) = v'( BY^{-m}) =  v'(B)-m v'(Y).
\end{equation}
Now $\deg_Y B \le q-1$  so  the ultrametric inequality
implies 
$v'(B) \ge (q-1) v'(Y)$.
We insert this last inequality into (\ref{ZBLEq2}) to find
\begin{equation}
\label{ZBLIneq3}
v'(A) \ge (q-1-m)v'(Y) \ge -v'(Y) > 0 
\end{equation}
because $q \le m$. Hence
\begin{alignat}1
\nonumber
\sum_{v'\in \places{F}} \max\{0,v'(A)\} &\ge \max\{0,v(A)\}
+ \sum_{\substack{v'\in \places{F} \\ v'(Y)<0}} \max\{0,v'(A)\} \\
\label{ZBLIneq4}
& \ge v(A)+ (m+1-q)\sum_{\substack{v'\in \places{F} \\ v'(Y) <0}}\max\{0,-v'(Y)\}
\end{alignat}
where the last inequality follows from (\ref{ZBLIneq3}).
Next we insert the equality 
\begin{equation*}
\sum_{\substack{v'\in \places{F} \\ v'(Y)<0}} \max\{0,v'(Y^{-1})\} =
[F:\IQbar(Y^{-1})] = [F:\IQbar(Y)]=p
\end{equation*}
 into (\ref{ZBLIneq4}) to find
\begin{equation}
\label{ZBLIneq5}
\sum_{v'\in \places{F}} \max\{0,v'(A)\} \ge v(A) + (m+1-q)p > 0.
\end{equation}
In particular $A\notin\IQbar$.

We continue by bounding the left-hand side of (\ref{ZBLIneq5}) from above.
If $v'\in \places{F}$ with $v'(X)\ge 0$ and $v'(Y)\ge 0$ then $v'(A)\ge 0$ because
$A$ is a polynomial in $X$ and $Y$. Hence
\begin{alignat}1
\label{ZBLIneq6}
  \begin{aligned}    
\sum_{v'\in \places{F}} \max\{0,-v'(A)\} \le &\sum_{v'\in \places{F},v'(X) <
0}\max\{0,-v'(A)\}+ 
\\ 
&\sum_{v'\in \places{F}, v'(Y) < 0} \max\{0,-v'(A)\}.
  \end{aligned}
\end{alignat}
Actually equality holds above because at most one $v'(X)$, $v'(Y)$ can be
negative, but this is not important here.
Around (\ref{ZBLIneq3}) we showed that if $v'(Y)<0$ then $v'(A)>0$, hence the second term on the
right-hand side of (\ref{ZBLIneq6}) is zero. 
Now recall that $\deg_X A\le n$; if $v'(X)<0$ then $v'(Y)\ge 0$ and 
 the ultrametric inequality 
leads us to $v'(A) \ge n v'(X)$. If we insert this inequality into
(\ref{ZBLIneq6}) we get
\begin{alignat}1
\label{ZBLIneq7}
\sum_{v'\in \places{F}} \max\{0,-v'(A)\} &\le
n 
\sum_{v'\in \places{F}, v'(X)<0}
\max\{0,-v'(X)\}
\\ \nonumber
&=n[F:\IQbar(X^{-1})]= n[F:\IQbar(X)] = nq.
\end{alignat}

The left-hand sides of (\ref{ZBLIneq5}) and (\ref{ZBLIneq7}) are both
equal to $[F:\IQbar(A)]$, so we get
\begin{equation}
\nonumber
v(A) \le nq + (q-m-1)p = t+pq-p-q.
\end{equation}
If we set $s=v(A)$, then Lemma \ref{HopitalLemma} and induction give
$v(D^k(A))=v(A)-k$ for $0\le k \le s$. Hence $D^s(A)(\pi)\not= 0$ 
and $D^k(A)(\pi)=0$ for $0\le k <s$.
\end{proof}

\section{Completion of Proof}

\begin{lemma}
\label{regLemma}
Say $\kappa > 0$ and $\lambda >0$ are real numbers with
$\kappa\lambda\ge 2$. 
Let $P\in\IQbar[X,Y]$ be irreducible with $p=\deg_X P\ge 1$, $q=\deg_Y
P\ge 1$,
and $\deg P = p+q$. 
For a regular zero $(x,y)\in\IQbar^2$  of $P$ we define
\begin{equation}
\label{def:kh}
  k = \max\left\{\frac{\height{x}}{q},\frac{\height{y}}{p}\right\} 
\quad\text{and}\quad
h=\log\left((p+1)(q+1)\right)+\hproj{P}.
\end{equation}
If $k\ge \lambda^2 h$,  then
\begin{alignat*}1
h(y) \le &\frac pq h(x) 
+ p\left(\kappa + \frac{1}{\lambda}+\frac{4}{\kappa}\right) 
(hk)^{1/2}
\\
&+\frac{1}{\kappa\lambda}\left( 
\frac{\log(16\kappa\lambda)}{8} + \frac{\log 3}{24}+
9\log(p+1) + \log\left(1+\frac{\kappa\lambda}{2}\right)\right).
\end{alignat*}
\end{lemma}
\begin{proof}
We first handle the case $q=1$ using Lemma \ref{AuxLemma}(iii).
Indeed, as $P$ is irreducible, it cannot be divisible by a polynomial
that is linear in $X$. 
As $h\ge \hproj{P}$ and $h\ge \log(p+1)$, we find
\begin{alignat}1
\label{eq:qequals1bound}
 h(y)\le p h(x) +  \log(p+1)+ \hproj{P}
\le ph(x) + h \le p h(x) + \lambda^{-1} ({hk})^{1/2}
\end{alignat}
where we also used  $h\le \lambda^{-2}k$. 
Observe that $\lambda^{-1} \le p\kappa$ because $\kappa\lambda \ge
1$. The
 bound (\ref{eq:qequals1bound}) is better than our claim. Hence we now
assume $q\ge 2$.

As $h>0$ we may define
\begin{equation}
\label{definemn}
m= q \left\lceil \kappa pq \left(\frac kh\right)^{1/2}\right\rceil \quad\text{and}\quad n=m\frac pq + p -1
\end{equation}
where $\lceil z\rceil$ is the least integer greater or equal to
$z\in\IR$. 
So $m$ and $n$ are integers and 
\begin{equation}
\label{eq:mtlb}
   m \ge \kappa pq^2 \left(\frac kh\right)^{1/2} \ge \kappa \lambda
  pq^2
\end{equation}
and
\begin{equation}
\label{eq:mtub}
   m \le  q\left(\kappa pq \left(\frac
  kh\right)^{1/2}+1\right)
= \kappa pq^2 \left(\frac
  kh\right)^{1/2} + q. 
\end{equation}
The lower bound for $m$ implies $m> 2q$ as
$\kappa\lambda> 1$ by hypothesis and since $q\ge 2$. Therefore, $m\ge
2q+1$ and $n\ge p$. 

Let $t=q(n+1)-mp$ be as in
Lemma \ref{ConstructLemma1}, then $t=pq\ge 1$. 
So  the said lemma provides $A,B\in\IQbar[X,Y]$ as therein.
Because of Lemma \ref{ZeroBoundLemma}
there exists an integer $s$ with 
\begin{equation*}
 0\le s \le t+pq-p-q \le 2pq -1 
\end{equation*}
 such that 
$D^s(A)(x,y)\not = 0$ and $D^k(A)(x,y)=0$ for all $0\le k < s$.
We apply Leibniz's rule (\ref{eq:leibniz}) to $D^s(AY^m-B)$ and the use
the fact that $AY^m-B$ is divisible by $P$ to conclude
\begin{equation}
\nonumber
y^m = \frac{D^s(B)(x,y)}{D^s(A)(x,y)}.
\end{equation}
 Lemma \ref{AuxLemma}(ii) and $h(y^m)=mh(y)$ gives
\begin{alignat}1
\label{hybound1}
m\height{y} &\le \hproj{D^s(A),D^s(B)}
+ \max\{\deg_X D^s(A),\deg_X D^s(B)\}\height{x} \\ \nonumber
 &\quad+ \max\{\deg_Y D^s(A),\deg_Y D^s(B)\}\height{y}  \\ \nonumber
  &\quad+\log\max\{(\deg_X D^s(A)+1)(\deg_Y D^s(A)+1),
\\ \nonumber
&\phantom{{}\quad+\log\max\{}(\deg_X D^s(B)+1)(\deg_Y D^s(B)+1)\}.
\end{alignat}
We use Lemma \ref{DsBoundLemma} applied to $A$ and $B$ to deduce
\begin{alignat}1
\nonumber
\max\{\deg_X D^s(A), \deg_X D^s(B)\} &\le n+(p-1)s,\\
\nonumber
\max\{\deg_Y D^s(A), \deg_Y D^s(B)\} &\le (q-1)(s + 1)
\end{alignat}
and
\begin{equation*}
\hproj{D^s(A),D^s(B)} \le \hproj{A,B} +s\hproj{P}+ s\log\bigl(2(p+1)(q+1)r(n+rs)\bigr)
\end{equation*}
with $r=\max\{p,q\}$; the last line
follows from summing up the local bounds in (\ref{DerivBound2}).
We insert these bounds in (\ref{hybound1}) to see
\begin{alignat}1
\nonumber
m\height{y} &\le \hproj{A,B} + s\hproj{P} +
(n+(p-1)s)\height{x} + (q-1)(s+1)\height{y}
\\ \nonumber
& 
+ s\log \bigl(2(p+1)(q+1)r(n+rs)\bigr) 
+\log\bigl((1+n+(p-1)s)q(s+1)\bigr).
\end{alignat}
Next we use the bound given for $\hproj{A,B}$ in Lemma \ref{ConstructLemma1} 
and recall $t=pq$ to get
\begin{alignat}1
\label{hybound2}
  m\height{y} &\le \frac {m(n-p+1)}{pq} \left(\log\left((p+1)(q+1)\right)+\hproj{P}\right) +\frac{\log(2nq)}{2}
\\  \nonumber
& \quad + s\hproj{P}+ (n+(p-1)s)\height{x} 
 + (q-1)(s+1)\height{y}\\
& \quad +s\log \bigl(2(p+1)(q+1)r(n+rs)\bigr)  
 +\log\bigl((1+n+(p-1)s)q(s+1)\bigr).
\nonumber
\end{alignat}
We use $n \le mp/q + p$, which
follows from (\ref{definemn}),  and find
\begin{alignat*}1
  (n+(p-1)s)h(x) &+ (q-1)(s+1)h(y) \\
&\le
m\frac pq h(x)+
  (p+(p-1)s)q k + (q-1)(s+1)pk
\\
&\le m\frac pq h(x)+\bigl(pq + (p-1)sq + (q-1)(s+1)p\bigr)k\\
&\le m\frac pq h(x)+4p^2q^2k
\end{alignat*}
using  the definition of $k$ and since $s\le 2pq-1$. 
We insert this bound into  (\ref{hybound2}), divide by $m$, and use the definition of $h$ 
to get
\begin{alignat*}1
\height{y} &
\le  \frac pq h(x) +\left(\frac {n-p+1}{pq} +\frac{s}{m}\right)h + \frac{4p^2q^2}{m}k 
 +
\frac{\log(2nq)}{2m} \\
&+ \frac{2pq-1}{m} \log \bigl(2(p+1)(q+1)r(n+rs)\bigr) 
+\frac 1m \log\bigl((1+n+(p-1)s)q(s+1)\bigr)
\end{alignat*}

Suppose for the moment $s\ge 1$, then 
 $1+n+(p-1)s\le n+rs$. We also 
have  $q(s+1)\le 2pq^2 \le 2(p+1)(q+1)r$. 
Therefore, $(1+n+(p-1)s)q(s+1) \le 2(p+1)(q+1)r(n+rs)$
and this inequality also holds for $s=0$ as $n\ge 1$. 
We find
\begin{alignat}1
  \begin{aligned}
       \label{hybound3}
\height{y} &
\le  \frac pq h(x) +\left(\frac {n-p+1}{pq}+\frac sm\right) h + \frac{4p^2q^2}{m}k 
 +
\frac{\log(2nq)}{2m} \\
&\quad+ \frac{2pq}{m} \log \bigl(2(p+1)(q+1)r(n+rs)\bigr) 
  \end{aligned}
\end{alignat}

We now continue by bounding each term on the right-hand side of inequality
(\ref{hybound3}). 

The first term $\frac pq h(x)$ is the main contribution to $h(y)$. 

To bound the second term we recall
our choice  (\ref{definemn}) which 
yields 
$(n-p+1)/(pq)= m/q^2$.
Observe that
\begin{equation*}
  \frac {n-p+1}{pq}+\frac sm\le
\frac{m}{q^2} 
+\frac{2}{\kappa\lambda q}
\le \frac{m}{q^2} + \frac{1}{q}
\end{equation*}
by (\ref{eq:mtlb}), $s\le 2pq,$ and $\kappa\lambda\ge 2$.
So the second term on the right of (\ref{hybound3}) satisfies
\begin{alignat*}1
\left(\frac {n-p+1}{pq}+\frac sm\right)h \le \left(\frac {m}{q^2} + \frac
1q\right)h  
\le \left(\kappa p \left(\frac kh\right)^{1/2} + \frac 2q\right)h
\le \kappa p (kh)^{1/2} + h
\end{alignat*}
because of (\ref{eq:mtub}) and $q\ge 2$. Using
$h\le \lambda^{-2}k$ we find
\begin{equation}
\label{term2bound} 
  \left(\frac{n-p+1}{pq}+\frac sm\right)h \le p\left(\kappa + \frac{1}{p\lambda}
  \right)(kh)^{1/2}
\le p\left(\kappa + \frac{1}{\lambda} \right)(kh)^{1/2}.
\end{equation}

The third term in (\ref{hybound3}) can be bounded from above using 
the first inequality in (\ref{eq:mtlb})  as follows
\begin{equation}
\label{term3bound}
\frac{4p^2q^2}{m} k \le \frac{4p}{\kappa} (hk)^{1/2}. 
\end{equation}

We move on to the fourth term and recall $n\le  mp/q + p$.
As $m\ge \kappa\lambda pq^2\ge
4\kappa\lambda p\ge 4\kappa\lambda, m\ge q$, since $z\mapsto z^{-1}\log(4z)$ is decreasing on
$[1,\infty)$, and because $p^{-1}\log p\le (\log 3)/3$ we find
\begin{alignat}1
\label{term4bound}
\begin{aligned}
\frac{\log(2nq)}{2m} &\le \frac{\log(2mp + 2pq)}{2m}
\le \frac{\log(4mp)}{2m}
\le \frac{\log(16\kappa\lambda)}{8\kappa\lambda} + \frac{\log
  p}{8\kappa\lambda p} \\
&\le \frac{\log(16\kappa\lambda)}{8\kappa\lambda} + \frac{\log
  3}{24\kappa\lambda}. 
\end{aligned}
\end{alignat}

For the fifth term we use $ n\le mp/q +
p, s\le 2pq-1,$ and $q\ge 2$ to bound
\begin{alignat*}1
n+rs \le m\frac pq + p + r(2pq-1)  \le m \frac p2 + 2pqr
\le 2pqr \left(1+\frac{m}{4pq}\right)
\end{alignat*}
and thus
\begin{alignat*}1
  \frac{2pq}{m} \log(n+rs) &\le \frac{2\log(2pqr)}{\kappa\lambda q} +  \frac {2pq}{m}
  \log\left(1+  \frac{m}{4pq}\right) \\
&\le \frac{2\log (2pqr)}{\kappa\lambda q} + \frac{1}{\kappa\lambda}
\log\left(1+\frac{\kappa\lambda}{2}\right)
\end{alignat*}
as $m/(pq)\ge \kappa\lambda q\ge 2\kappa\lambda$ and since $z\mapsto z^{-1}\log(1+z/2)$ is
decreasing on $(0,\infty)$. We deduce 
\begin{alignat}1
\label{eq:3pqmbound}
  \frac{2pq}{m}\log(2(p+1)(q+1)r(n+rs)) 
\le \frac{2}{\kappa\lambda}\Bigl(&\frac{\log(2(p+1)r)}{q}
+\frac{\log(q+1)}{q} 
\\ \nonumber
&+
\frac{\log(2pqr)}{q} 
+ \frac 12 \log\Bigl(1+\frac{\kappa\lambda}{2}\Bigr)\Bigr). 
\end{alignat}

Below we will use the estimates $q^{-1}\log(q+1)\le (\log 3)/2$ and
$q^{-1}\log (2q) \le \log 2$, which hold since $q\ge 2$.
 
If $q> p$, then $r=q$ and so $q^{-1}\log (2(p+1)r)\le 0.5\log(p+1) + \log
2$. We have $q^{-1}\log (2q^2)\le \log \sqrt 8$. Hence the right-hand side
of (\ref{eq:3pqmbound}) is at most
\begin{alignat}1
\nonumber
  &\frac{2}{\kappa\lambda}\left(\log(p+1) + \log 2 + \frac 12
  \log 3 + \log \sqrt 8 + \frac 12
  \log\left(1+\frac{\kappa\lambda}{2}\right)\right)\\
\label{eq:3pqmbound2}
&\qquad \le
\frac{1}{\kappa\lambda}\left(9\log(p+1) 
+
\log\left(1+\frac{\kappa\lambda}{2}\right)\right).
\end{alignat}

Now suppose $p\ge q$, then $r=p\ge 2$. 
We bound $0.5 \log (2(p+1)p) \le 0.5\log(2)+\log(p+1) \le
1.5\log(p+1)$, 
$q^{-1}\log(q+1)\le 0.5 \log(p+1)$, and
$q^{-1} \log (2pqr) \le  0.5 \log(p^2)+\log 2 \le 2\log(p+1)$. 
In this case, 
 the right-hand side of (\ref{eq:3pqmbound}) is at
most
\begin{alignat*}1
\frac{2}{\kappa\lambda}\left(4\log(p+1) + 
\frac 12 \log\left(1+\frac{\kappa\lambda}{2}\right)\right)
= \frac{1}{\kappa\lambda}\left(8\log(p+1) + 
\log\left(1+\frac{\kappa\lambda}{2}\right)\right).
\end{alignat*}

In both cases we find that the fifth term is bounded by
(\ref{eq:3pqmbound2}). 
We insert this  upper bound as well as (\ref{term2bound}), (\ref{term3bound}),
(\ref{term4bound})  into (\ref{hybound3}) to find 
\begin{alignat*}1
\height{y} &\le \frac pq \height{x}
+ p\left(\kappa+\frac{1}{\lambda}  +\frac{4}{\kappa}\right) ({kh})^{1/2}
+\frac{1}{8\kappa\lambda}\left(\log(16\kappa\lambda) + \frac{\log 3}{3}\right)\\
&\qquad +
\frac{1}{\kappa\lambda}\left(9 \log(p+1) + \log\left(1+\frac{\kappa\lambda}{2}\right)\right),
\end{alignat*}
as required.
\end{proof}

\begin{lemma}
\label{SingLemma}
Let $P\in \IQbar[X,Y]$ be irreducible with $p=\deg_X P\ge 1$ and
$q=\deg_Y P\ge 1$.
If $(x,y)\in\IQbar^2$ with $P(x,y)=0$ is not a regular zero of $P$ then
\begin{equation}
\label{eq:singbound}
\max\left\{\frac{h(x)}{q},\frac{\height{y}}{p}\right\}
 \le  2\hproj{P}+4\log((p+1)(q+1)).
\end{equation}
\end{lemma}
\begin{proof} 
By symmetry we may assume $h(y)/p  \ge h(x)/q$. 
Let $D\in\IQbar[Y]$ be the resultant of the two polynomials
  $P,{\partial P}/{\partial X}\in\IQbar(Y)[X]$
(cf. Chapter IV, \S 8 \cite{LangAlgebra}). Then $D\not=0$ because $P$
  is irreducible in $\IQbar(Y)[X]$. The resultant $D$ is the determinant of a
  $(2p-1)\times (2p-1)$ matrix whose entries, denoted here by $m_{ij}$, are polynomials in $Y$ with
  degrees bounded by $q$. We find
 $\deg D \le (2p-1)q\le 2pq$ and, using the structure of the resultant
  matrix, that $D$ is a sum of at most
$(p+1)^{p-1}p^p \le (p+1)^{2p-1}$ products of the $m_{ij}$. If $K$ is a 
number field containing the coefficients of the $m_{ij}$ and $v\in \places{K}$, then
by Lemma \ref{AuxLemma}(i) we find
\begin{equation}
\label{Dbound}
|D|_v \le \delta_v{((p+1)^{2p-1})} \max_\sigma \{|m_{1,\sigma(1)}\cdots m_{2p-1,\sigma(2p-1)}|_v\}
\end{equation}
where $\sigma$ runs over all permutations of the first $2p-1$ positive
integers. The second bound of Lemma \ref{AuxLemma}(i) applied to the
univariate $m_{ij}$ yields 
\begin{equation}
\nonumber
|m_{1,\sigma(1)}\cdots m_{2p-1,\sigma(2p-1)}|_v \le
  \delta_v((q+1)^{2p-2})|m_{1,\sigma}|_v\cdots |m_{2p-1,\sigma(2p-1)}|_v. 
\end{equation}
This 
  inequality and $|m_{ij}|_v\le \delta_v(p)|P|_v$ 
inserted into (\ref{Dbound}) gives
\begin{equation}
\nonumber
|D|_v \le \delta_v{((p+1)^{2p-1} (q+1)^{2p-2} p^{2p-1})}|P|_v^{2p-1}
\le \delta_v((p+1)^{4p} (q+1)^{2p})|P|_v^{2p-1}.
\end{equation}
We take the  sum over all places of $K$ to find
 $\hproj{D} \le 2p\hproj{P}+4p\log (p+1)+2p\log(q+1)$. 

Now if $(x,y)\in\IQbar^2$ with $P(x,y)=0$ is not a regular zero of $P$, 
then $D(y)=0$.
By Lemma \ref{AuxLemma}(iii) we can bound $\height{y} \le \hproj{D}
  + \log\deg D \le \hproj{D}+\log(2pq)$. The bound (\ref{eq:singbound})
follows from this inequality together with the bound for $\hproj{D}$
and since $p^{-1}\log (2pq) \le \log (2q)\le 2\log(q+1)$.
\end{proof}

\begin{lemma}
\label{lem:degreepplusq}
Say $\kappa > 0$ and $\lambda >0$ are real numbers  with $\kappa\lambda\ge 4$. 
Let $P\in\IQbar[X,Y]$ be irreducible with $p=\deg_X P\ge 1$, $q=\deg_Y
P\ge 1$,
and $\deg P = p+q$. 
For $(x,y)\in\IQbar^2$ with $P(x,y)=0$ we define
$k$ and $h$ as in (\ref{def:kh}). 
If $k\ge \lambda^2 h$,  then
\begin{alignat}1 
  \begin{aligned}    
\label{eq:hxqhyp1}
\left| \frac{h(x)}{q}-\frac{h(y)}{p}\right|
&\le 
 \left(\kappa  + \frac{1}{\lambda}+\frac{4}{\kappa}\right) 
(hk)^{1/2}
+ \frac{9}{8\kappa\lambda} \log(363 \kappa\lambda). 
  \end{aligned}
\end{alignat}
\end{lemma}
\begin{proof}
By symmetry we may suppose $h(y)/p\ge h(x)/q$. 
Let us assume that $(x,y)$ is a regular zero of $P$. 
We divide the bound in 
 Lemma \ref{regLemma} by $p$, use $p^{-1}\log(p+1) \le \log 2$,
 $1+\kappa\lambda/2 \le \kappa \lambda$, and 
\begin{equation*}
\frac{\log(16\kappa\lambda)}{8p} + \frac{\log 3}{24p} 
+9 \log 2 + \frac{\log(\kappa\lambda)}{p}\le 
\frac 98 \log\left(2^{76/9}  3^{1/27} \kappa\lambda\right)
\le \frac 98 \log\left(363 \kappa\lambda\right)
\end{equation*}
to conclude (\ref{eq:hxqhyp1}).

The upper bound given by Lemma 
\ref{SingLemma} is at most $4h\le 4\lambda^{-1} (hk)^{1/2}$.
If $(x,y)$ is not a regular zero. 
This is also an upper bound for $|h(x)/q-h(y)/p|$ so the current lemma
 follows as $4\lambda^{-1}\le \kappa$. 
\end{proof}

\begin{lemma}
\label{RationalTransformation}
Let $P\in \IQbar[X,Y]$ be 
irreducible with $p=\deg_X P \ge  1$ and $q=\deg_Y P $. Then
there is a root of unity $\xi$ such that the polynomial 
$\tilde P = X^p P(X^{-1}+\xi,Y)$
has total degree $p+q$, is irreducible in $\IQbar[X,Y]$, and
satisfies
\begin{equation}
\nonumber
\deg_X \tilde P = p,\quad \deg_Y \tilde P = q,\quad 
\hproj{\tilde P} \le p\log 2+\hproj{P}.
\end{equation}
\end{lemma}
\begin{proof} We may write $P=\sum_j a_j Y^j$ with 
$a_j=\sum_i a_{ij}X^i\in \IQbar[X]$. By
  hypothesis we have $a_q\not=0$ and $a_q$ has degree at most $p$ as a
  polynomial in $X$. We  choose a root of unity $\xi$ such that
  $a_q(\xi)\not=0$. A direct computation
using the irreducibility of $P$ shows that $\tilde P$ is irreducible.
By construction $\deg_X \tilde P \le p$, $\deg_Y \tilde P \le q$ and
  therefore $\deg \tilde P \le p+q$. 
 
Say $0\le i\le p$ and $0\le j \le q$, then the
coefficient of $X^iY^j$ in $\tilde P$ equals
\begin{equation}
\label{CoeffXiYj}
\sum_{k=p-i}^p a_{kj} 
{k \choose p-i}
\xi^{i-p+k}.
\end{equation}
So if $i=p$ and $j=q$ we see that the coefficient of $X^pY^q$ is $a_q(\xi)\not=0$ and
conclude that $\deg_X \tilde P = p$, $\deg_Y \tilde P = q$, and
$\deg\tilde P = p+q$. 

Now say $K$ is a number field that contains $\xi$ and
all coefficients of $P$.
If $v\in\places{K}$, then by (\ref{CoeffXiYj})
and standard facts on binomial coefficients we get
\begin{alignat*}1
|\tilde P|_v &\le \max_{0\le i\le p}\delta_v
\left(\sum_{k=p-i}^p 
{k \choose p-i}
\right)|P|_v
= \max_{0\le i\le p}\delta_v\left(
{p+1 \choose p-i+1}
\right) |P|_v 
\\ &\le \delta_v(2^p)|P|_v.
\end{alignat*}
We sum over the locals bounds to complete the proof. 
\end{proof}

\begin{proof}[of Theorem \ref{QuasiequivThm}]
By symmetry, we may suppose $p= \deg_X P \le \deg_Y P=q$. 

First suppose $q=1$. Then $p=1$.
Clearly, $|h(x)-h(y)|\le \max\{h(x),h(y)\}$
and Lemma \ref{AuxLemma}(iii)  implies
$|h(x)-h(y)|\le \log(2)+ \hproj{P}$. So $|h(x)-h(y)|$ is bounded from
above by the geometric mean
\begin{equation*}
  (\log(2)+\hproj{P})^{1/2} \max\{h(x),h(y)\}^{1/2}
\end{equation*}
which is less than  the bound in the assertion. 
 
So we may assume $q\ge 2$, in particular $(p+1)(q+1) \ge 6$. 
We make the choice $\kappa = 2.25$ and  $\lambda = 4.98$.
Let $k = \max\{h(x)/q,h(y)/p\}$.

We let $\xi$ and $\tilde P$ be as in 
 Lemma \ref{RationalTransformation}.

Say $(x,y)$ is as in the hypothesis. If $x=\xi$, then $h(x)=0$ and $h(y)\le  \log((p+1)q) + \hproj{P}$ by Lemma
\ref{AuxLemma}(iii). 
Thus
\begin{equation*}
  \left|\frac{h(x)}{q}-\frac{h(y)}{p}\right| =  \frac{h(y)}{p}
\le \bigl(\log((p+1)q) + \hproj{P}\bigr)^{1/2}
\left(\frac{h(y)}{p}\right)^{1/2}. 
\end{equation*}
This is better than our claim. 

We now assume  $x\not=\xi$ and define $\tilde x = (x-\xi)^{-1}$. 
Hence $\tilde P(\tilde x,y)=0$. By basic height properties, cf.
Lemma 1.5.18 and Proposition 1.5.15 \cite{BG}, 
we find $h(\tilde x) = h(x-\xi)$ and 
$|h(\tilde x)-h(x)|\le \log 2$. Thus
\begin{equation}
\label{eq:hxtilde}
\left|  \frac{h(x)}{q}-\frac{h(y)}{p} \right|
\le \left|  \frac{h(\tilde x)}{q}-\frac{h(y)}{p} \right|+ \frac{\log
  2}{q}
\le \left|  \frac{h(\tilde x)}{q}-\frac{h(y)}{p} \right|+ \frac{\log 2}{2}.
\end{equation}

First, we suppose $k-(\log 2)/q\ge
\lambda^2\tilde h$ with $\tilde h=  \log\left((p+1)(q+1)\right)+ \hproj{\tilde P}$.
Then $\max\{h(\tilde x)/q,h(y)/p\} \ge \lambda^2\tilde h$.
 Lemma \ref{lem:degreepplusq} applied to $\tilde P$ and (\ref{eq:hxtilde}) imply
\begin{alignat*}1
  \left|\frac{h(x)}{q}-\frac{h(y)}{p}\right|
&\le 
 \left(\kappa+ \frac{1}{\lambda}+\frac{4}{\kappa} \right) 
{(\tilde h\tilde k)}^{1/2}
+ \frac{9}{8\kappa\lambda} \log\left(363\kappa\lambda\right)
+\frac{\log 2}{2} \\
&\le 4.229 ({\tilde h\tilde k})^{1/2}  + 1.181
\end{alignat*}
with our choice of $\kappa$ and $\lambda$
and with  $\tilde k  = \max\{h(\tilde x)/q,h(y)/p\}$.
Now 
\begin{equation}
\label{eq:htildeub}
  \tilde h \le \log\left((p+1)(q+1)\right)+ p\log 2+\hproj{P}
= \log\left(2^{p}(p+1)(q+1)\right) +\hproj{P}.
\end{equation}
We have $k\ge \lambda^2 \tilde h \ge \lambda^2\log 6 > 1$
and therefore
\begin{equation*}
  \tilde k\le   k + \frac{\log 2}{q}\le k+\frac{\log 2}{2}\le  k + \frac{\log 2}{2 \lambda^2
    \log 6} k \le 1.008k. 
\end{equation*}
 We conclude
$4.229 ({\tilde h\tilde k})^{1/2} + 1.181 \le 
4.246 (\log(2^{p}(p+1)(q+1)) +\hproj{P})^{1/2}  k^{1/2}+1.181$. 
Observe that $\log\left(2^{p}(p+1)(q+1)\right) \ge \log 12$. The
theorem follows in this case as $k\ge 1$ and
$4.246 + 1.181 (\log 12)^{-1/2} < 5$. 

Second, we must treat the case $k-(\log 2)/q< \lambda^2\tilde h$.
The bound (\ref{eq:htildeub}) continues to hold and
we find
\begin{alignat*}1
  \left|\frac{h(x)}{q}-\frac{h(y)}{p}\right|
&\le \max\left\{\frac{h(x)}{q},\frac{h(y)}{p}\right\} =k
\le \left( \frac{\log 2}{2} + \lambda^2 \tilde h\right)^{1/2}k^{1/2}
\\ 
&\le \left(\frac{\log 2}{2} + \lambda^2\log(2^{p}(p+1)(q+1))+\lambda^2\hproj{P}\right)^{1/2}
k^{1/2} \\
&\le \left(\frac{\log 2}{2\log 12}+\lambda^2\right)^{1/2}  \left( \log(2^{p}(p+1)(q+1))+\hproj{P}\right)^{1/2}
k^{1/2}
\end{alignat*}
where we used $\log(2^p(p+1)(q+1))\ge\log 12$ again. 
This is better than our claim as $(\log(2)/(2\log 12)+\lambda^2)^{1/2}<5$. 
\end{proof}

Corollary \ref{QuasiequivCor} is easy to prove using
Theorem \ref{QuasiequivThm}. We set $k = \max\{h(x)/q,h(y)/p\}$. 

Say $P(x,y)=0$ with $x,y\in\IQbar$ and $\height{y}> 2\frac
pq\height{x}$.
Then
\begin{equation*}
\frac{k}{2}\le \max\left\{\frac{h(x)}{q},\frac{h(y)}{2p}\right\} = \frac{h(y)}{2p}<  \frac{h(y)}{p}-\frac{h(x)}{q}.
\end{equation*}
By Theorem \ref{QuasiequivThm} we have
\begin{equation}
\nonumber
\frac{k}{2}<
5 (\log(2^{r}(p+1)(q+1))+\hproj{P})^{1/2}
k^{1/2}
\end{equation}
and so 
\begin{equation*}
  k^{1/2} <10  \left( \log(2^{r}(p+1)(q+1)) +
  \hproj{P}\right)^{1/2}.
\end{equation*}
This contradicts the bound in (\ref{QuasiequivCorHypo}).

\section{On a Theorem of Runge}
In this section we will prove Theorem \ref{RungeThm}.

Let $P=\sum_{i,j} p_{ij}X^i Y^j\in\IZ[X,Y]$ be irreducible in
$\IQbar[X,Y]$ with $d=\deg P=\deg_X P = \deg_Y P$.
Furthermore, since Theorem \ref{RungeThm} applies only to polynomials of degree
at least $2$, we assume $d\ge 2$.
We use $P_d = \sum_{i+j=d} p_{ij}X^i Y^j$ to denote the homogeneous part of
degree $d$. Finally, we can factor $P_d$ as  
$p_{d0} \prod_{s=1}^d (X-t_s Y)$ with $t_s\in\IQbar^\times$
for each $s\in \{1,\ldots,d\}$.

\begin{lemma}
\label{apps:ConstructR}
Set $t=t_s$ for some $s\in \{1,\ldots,d\}$.
We have $\height{t}\le \hproj{P}+\log d$. Furthermore, if $R_t(X,Z)=
P(X,t^{-1}(X-Z))\in\IQbar[X,Z]$, then $\deg_Z R = d$, $\deg_X R \le d-1$,
and $\hproj{R_t}\le d(\log(2d)+2\hproj{P})$. 
\end{lemma}
\begin{proof}
The bound for $\height{t}$ follows from Lemma \ref{AuxLemma}(iii); 
indeed $t$ is a zero of $P_d(X,1) \in \IQbar[X]$ whose coefficients
are coefficients of $P$.

Let $i,j\ge 0$ be integers, the coefficient of $X^i Z^j$ in $R_t$ is
\begin{equation}
\nonumber
(-1)^j \sum_{k=0}^{d-j} 
{j+k \choose k}
p_{i-k,j+k}t^{-j-k}.
\end{equation}
Clearly, $\deg_Z R\le d$. 
The coefficient of $Z^d$ is nonzero and that of $X^i Z^j$ is zero
provided $i\ge d$. 
The lemma now follows from local inequalities as in the proof of
Lemma \ref{RationalTransformation} together with basic height properties.
\end{proof}

If $P(x,y)=0$ with $x,y$ algebraic numbers, 
then $R_t(x,x-ty)=0$. 

\begin{lemma}
\label{apps:hzSmall}
Let $P(x,y)=0$ with $x,y\in\IQ$ and
\begin{equation}
\label{apps:hxLowerBound}
h(x) \ge 100 d \left( \log\left((4d)^dd(d+1)\right) + 2d\hproj{P}\right).
\end{equation}
Set $t=t_s$ for
some $s\in \{1,\ldots,d\}$, 
there exists an embedding $\sigma:\IQ(t)\rightarrow\IC$
such that
\begin{equation}
\nonumber
\log\max\{1,|x-\sigma(t)y|\}\le \frac{d-1}{d} h(x) + 10.4 d
(\log(2d)+\hproj{P})^{1/2} h(x)^{1/2}.
\end{equation}
\end{lemma}
\begin{proof} 
Let $R=R_t$ be as in Lemma \ref{apps:ConstructR}, then
$R$ is irreducible. Now $\deg_X R \ge 1$, indeed if $\deg_X R =0$, 
then by construction $P$ has degree $1$, which contradicts our assumption
$d\ge 2$.
Let $z=x-ty$, so $R(x,z)=0$. 

We recall that $\deg_X R \le d-1$ and $\deg_Z R = d$. So
\begin{equation}
\label{eq:hRbound}
  \log(2^{\deg_X R }(\deg_X R + 1)(\deg_Z R + 1)) + \hproj{R}
\le \log\left((4d)^{d}d(d+1)\right) + 2d\hproj{P} 
\end{equation}
by Lemma \ref{apps:ConstructR}. Our hypothesis
(\ref{apps:hxLowerBound}) 
together with Corollary \ref{QuasiequivCor} yields
$h(z) \le 2 (\deg_X R) h(x)/d$ and hence
\begin{equation}
\label{eq:boundhxdhzdegX}
  \max\left\{\frac{h(x)}{d},\frac{h(z)}{\deg_X R} \right\}
\le 2\frac{h(x)}{d}.
\end{equation}

Next we apply Theorem \ref{QuasiequivThm}  to $R$ and use
(\ref{eq:boundhxdhzdegX}) as well as (\ref{eq:hRbound}) to find
\begin{equation*}
\frac{\height{z}}{\deg_X R}\le \frac{h(x)}{d}
+ \frac{5 \sqrt 2}{ d^{1/2}} \left(\log\left((4d)^{d}d(d+1)\right) + 2d\hproj{P}\right)^{1/2}h(x)^{1/2}.
\end{equation*}
We multiply with $\deg_X R \le d-1$ to obtain
\begin{equation}
\nonumber
\height{z}\le \frac{d-1}{d}\height{x}+10  d
\left(\frac{1}{2d} \log\left((4d)^{d}d(d+1)\right) + \hproj{P}\right)^{1/2}\height{x}^{1/2}. 
\end{equation}
As $d\ge 2$ we find
$(2d)^{-1} \log\left((4d)^dd(d+1)\right)\le 1.08 \log(2d)$ and this yields
\begin{equation}
\nonumber
\height{z}\le \frac{d-1}{d}\height{x}+10.4 d
\left(\log(2d) + \hproj{P}\right)^{1/2}\height{x}^{1/2}. 
\end{equation}

By the definition (\ref{def:weilheight}) of the height there exists
an embedding $\sigma:\IQ(t)\rightarrow\IC$ with 
$\log\max\{1,|x-\sigma(t)y|\}\le h(x-\sigma(t)y)=\height{z}$. This concludes the proof.
\end{proof}

We now prove Theorem \ref{RungeThm}:

Let $P$ be as in the hypothesis, so $d=\deg_X P \ge 2$ and there
 are non-conjugated zeros $t',t''\in\IQbar^\times$ 
 of $P_d(X,1)$.
Suppose $x,y\in\IZ$ with
$P(x,y)=0$. 
By symmetry we may assume $|y|\le |x|$ and  also $x\not=0$.

To prove this theorem, we may assume $h(x)=\log|x| \ge 100 d^4 h$ with $h =
\log(2d)+\hproj{P}$. 
Using $d\ge 2$ we find 
\begin{equation*}
  h 
\ge \frac{\log(4d)}{d^2} + \frac{\log(d(d+1))}{d^3} + \frac{2}{d^2} \hproj{P}
\end{equation*}
and this implies
\begin{equation*}
  h(x) \ge 100 d^4 h \ge 100 d\left(\log\left((4d)^dd(d+1)\right) + 2d\hproj{P}\right),
\end{equation*}
the hypothesis of Lemma \ref{apps:hzSmall}.

Let $\sigma',\sigma''$ be the embeddings 
given by Lemma \ref{apps:hzSmall} applied to $t',t''$ respectively,
then $\sigma'(t')\not=\sigma''(t'')$. For brevity we define
$\xi=x-\sigma'(t')y$ and $\eta=x-\sigma''(t'')y$; we eliminate $y$ to get
\begin{equation}
\nonumber
x= \frac{\xi \sigma''(t'')-\eta\sigma'(t')}{\sigma''(t'')-\sigma'(t')}.
\end{equation}
So
\begin{equation}
\label{apps:Boundx}
|x|\le 2\max\{|\xi|,|\eta|\}\max\{|\sigma'(t')|,|\sigma''(t'')|\}
|\sigma'(t')-\sigma''(t'')|^{-1}
\end{equation}
for the complex absolute value $|\cdot|$.

We will bound $|\xi|$ and $|\eta|$ using Lemma \ref{apps:hzSmall}, but first we
bound the remaining absolute values in (\ref{apps:Boundx}) using height
inequalities. If $\alpha\in\IQbar$, then
$\log\max\{1,|\alpha|_v\}\le [\IQ(\alpha):\IQ]\height{\alpha}$
for any $v\in\places{\IQ(\alpha)}$. This inequality follows immediately from
the definition of the height. Since for example 
$[\IQ(t'):\IQ]\le d$ and
$[\IQ(t',t''):\IQ]\le d^2$ we deduce from
(\ref{apps:Boundx}) that
\begin{equation}
\label{Boundlogabsx}
\log|x|\le \log 2 +  d \max\{\height{t'},\height{t''}\} + d^2
\height{\sigma'(t')-\sigma''(t'')}+\log\max\{1,|\xi|,|\eta|\}.
\end{equation}

Lemma \ref{apps:ConstructR} implies 
$\max\{\height{t'},\height{t''}\} \le \hproj{P}+\log d$.
Next we use the bounds for $|\xi|$ and $|\eta|$ from  Lemma \ref{apps:hzSmall}.
Together with (\ref{Boundlogabsx}) and
 standard height inequalities we have
\begin{alignat}1
\nonumber
\log|x|&\le \log 2+ d( \hproj{P}+\log d) +
 d^2(2\hproj{P}+2\log d + \log 2)
\\ \nonumber
& \qquad +\frac{d-1}{d}h(x)  + 10.4 d h^{1/2}h(x)^{1/2}
\\ \nonumber
& \le  \frac{d-1}{d}h(x)+3d^2 h + 10.4 dh^{1/2}h(x)^{1/2}.
\end{alignat}
As $h(x)=\log|x|$ we find
\begin{equation}
\nonumber
\log|x|\le 3d^3h + 10.4 d^{2}h^{1/2}(\log|x|)^{1/2}.
\end{equation}

If $B,C,T\ge 0$ with $T\le C+B\sqrt{T}$, then
\begin{equation*}
T \le \frac 14 \left(\left(B^2+4C\right)^{1/2} + B\right)^2. 
\end{equation*}
We apply this inequality with 
$B=10.4d^2 h^{1/2}$, $C=3d^3h$, and
$T=\log|x|$ to conclude 
that 
\begin{equation*}
  \log |x| \le \frac 14 \left (\left(10.4^2 + 12\right)^{1/2}d^2 h^{1/2}
+ 10.4d^2 h^{1/2}\right)^2 \le 
115 d^4 h.
\end{equation*}
 The upper bound for $\log|y|$ in (\ref{RungeThm}) follows
since $|y|\le |x|$.
\bibliographystyle{amsplain}
\bibliography{literature}
\end{document}